\documentclass[letterpaper,10pt,reqno]{amsart}
\usepackage{graphicx}

\usepackage[theoremfont,largesc]{newtxtext}
\usepackage[]{newtxmath}
\usepackage{microtype}

\numberwithin{equation}{section}

\newtheoremstyle{plain}
  {\medskipamount}   
  {\medskipamount}   
  {\slshape}  
  {0pt}       
  {\bfseries} 
  {.}         
  {5pt plus 1pt minus 1pt} 
  {}          

\theoremstyle{plain}
\newtheorem{thm}{Theorem}[section]

\newtheorem{lem}[thm]{Lemma}
\newtheorem{prop}[thm]{Proposition}
\newtheorem{cor}[thm]{Corollary}

\newtheorem{quest}{Question}

\theoremstyle{definition}
\newtheorem{defn}[thm]{Definition}

\theoremstyle{remark}
\newtheorem*{rem}{Remark}

\newcommand{\NN}{\mathbb{N}}
\newcommand{\ZZ}{\mathbb{Z}}

\newcommand{\RR}{\mathbb{R}}
\newcommand{\CC}{\mathbb{C}}

\let\Re\relax
\DeclareMathOperator{\Re}{Re}

\DeclareMathOperator*{\Res}{Res}

\begin{document}

\title{Dirichlet series associated to sum-of-digits functions}
\author{Corey Everlove}

\thanks{Work partially supported by NSF grant DMS-1701576.}

\date{\today}

\maketitle

\begin{abstract}
We study the Dirichlet series $F_b(s)=\sum_{n=1}^\infty d_b(n)n^{-s}$, where $d_b(n)$ is the sum of the base-$b$ digits of the integer $n$, and $G_b(s)=\sum_{n=1}^\infty S_b(n)n^{-s}$, where $S_b(n)=\sum_{m=1}^{n-1}d_b(m)$ is the summatory function of $d_b(n)$. We show that $F_b(s)$ and $G_b(s)$ have continuations to the plane $\CC$ as meromorphic functions of order at least 2, determine the locations of all poles, and give explicit formulas for the residues at the poles. We give a continuous interpolation of the sum-of-digits functions $d_b$ and $S_b$ to non-integer bases using a formula of Delange, and show that the associated Dirichlet series have a meromorphic continuation at least one unit left of their abscissa of absolute convergence.
\end{abstract}

\section{Introduction}

There has been a great deal of study of properties of the radix expansions to an integer base $b\geq 2$ of integers $n$. For each integer base $b\geq 2$, every positive integer $n$ has a unique base-$b$ expansion
\begin{equation}
n = \sum_{i \geq 0} \delta_{b,i}(n)b^i
\end{equation}
with digits $\delta_{b,i}\in\{0,1,\dotsc,b-1\}$ given by
\begin{equation}
\delta_{b,i} = \Bigl\lfloor \frac{n}{b^i} \Bigr\rfloor - b \Bigl\lfloor \frac{n}{b^{i+1}} \Bigr\rfloor.
\end{equation}
This paper considers two summatory functions of base $b$ digits of $n$:

\begin{enumerate}
\item The \emph{base-$b$ sum-of-digits function} $d_b(n)$ is
\begin{equation}
d_b(n) = \sum_{i\geq 0} \delta_{b,i}(n).
\end{equation}

\item The \emph{(base $b$) cumulative sum-of-digits function} $S_b(n)$ is
\begin{equation}
S_b(n) = \sum_{m=1}^{n-1} d_b(m).
\end{equation}
We follow here the convention of previous authors (including \cite{delange-75} and \cite{flajolet-94}), with the sum defining $S_b(n)$ running to $n-1$ instead of $n$.
\end{enumerate}

We associate to the functions $d_b(n)$ and $S_b(n)$ the Dirichlet series generating functions
\begin{equation}
F_b(s) = \sum_{n=1}^\infty \frac{d_b(n)}{n^s}
\end{equation}
and
\begin{equation}
G_b(s) = \sum_{n=1}^\infty \frac{S_b(n)}{n^s}.
\end{equation}
These Dirichlet series have abscissa of convergence $\Re(s)=1$ and $\Re(s)=2$, respectively.

This paper studies the problem of the meromorphic continuation to $\CC$ of Dirichlet series associated 
to the base-$b$ digit sums $d_b(n)$ and $S_b(n)$. 
Here we obtain the meromorphic continuation  and determine
its exact pole and residue structure. 
The pole structure contains half of a two-dimensional lattice
and the residues involve Bernoulli numbers and values of the Riemann zeta function on
the line $\Re(s)=0$. A  meromorphic continuation of these functions was  previously obtained in 
the thesis of Dumas \cite{dumas-thesis} by a different method, which specified a
half-lattice containing all the poles but did not determine the residues; in fact infinitely many
of the residues on his possible pole set vanish.

The asymptotics of $S_b(n)$ have been extensively studied, see Section \ref{sec-previous-work}. 
We mention particularly  work of Delange \cite{delange-75}, 
given below as Theorem \ref{thm-delange},
which gives an exact formula for $S_b(n)$ in terms of a continuous nondifferentiable function with Fourier coefficients involving values of the Riemann zeta function on the imaginary axis.
 Using an interpolation of Delange's formula we formulate a 
continuous interpolation of $S_b(n)$ in the base parameter $b$,
permitting definitions of  $d_\beta(n)$ and $S_\beta(n)$ for a real parameter $\beta>1$. We obtain a meromorphic continuation of the associated Dirichlet series $F_{\beta}(s)$
and $G_{\beta}(s)$ to the half-planes $\Re(s)>0$ and $\Re(s)>1$, respectively. We note apparent
fractal properties of $d_{\beta}(n)$ as $\beta$ is varied. 

\subsection{Results}

 Our first results concern the meromorphic continuation of the functions $F_b(s)$ and $G_b(s)$ to the entire complex plane $\CC$.

\begin{thm}\label{thm-db}
For each integer base $b\geq 2$, the function $F_b(s) = \sum_{n=1}^\infty d_b(n)n^{-s}$ has a meromorphic continuation to $\CC$. The poles of $F_b(s)$ consist of a double pole at $s=1$ with Laurent expansion beginning
\begin{equation}
F_b(s)=\frac{b-1}{2\log b}(s-1)^{-2} + \biggl( \frac{b-1}{2\log b} \log(2\pi) - \frac{b+1}{4}\biggr)(s-1)^{-1} + O(1),
\end{equation}
simple poles at each other point $s=1+2\pi i m / \log b$ with $m\in \ZZ$ ($m\neq 0$) with residue
\begin{equation}
\Res\biggl( F_b(s) , s=1+\frac{2\pi i m}{\log b} \biggr) = - \frac{b-1}{2\pi i m} \zeta\biggl(\frac{2\pi i m}{\log b}\biggr),
\end{equation}
and simple poles at each point $s=1-k+2\pi i m /\log b$ with $k=1$ or $k\geq 2$ an even integer and with $m\in\ZZ$, with residue
\begin{equation}
\Res\biggl( F_b(s) , s=1-k+\frac{2\pi i m}{\log b} \biggr) =  (-1)^{k+1}\frac{b-1}{\log b}\zeta\biggl(\frac{2\pi i m}{\log b}\biggr)\frac{B_k}{k!} \prod_{j=1}^{k-1} \biggl(\frac{2\pi i m}{\log b} - j\biggr)
\end{equation}
where $B_k$ is the $k$th Bernoulli number.
\end{thm}

Theorem \ref{thm-db} is proved by first considering the Dirichlet series $\sum \bigl(d_b(n)-d_b(n-1)\bigr) n^{-s}$ and then exploiting a relation between power series and Dirichlet series to recover $F_b(s)$. The proof is presented in Section \ref{sec-mero-cont}.

The meromorphic continuation of Dirichlet series attached to $b$-regular sequences, of which our Dirichlet series $F_b(s)$ is a particular example, was studied by Dumas in his 1993 thesis \cite{dumas-thesis}; this work also showed that the poles of $F_b(s)$ must be contained in a certain half-lattice, strictly larger than the half-lattice here.

A similar method allows us to meromorphically continue the series $G_b(s)$ to the complex plane.

\begin{thm}\label{thm-sb}
For each integer $b\geq 2$, the function $G_b(s)=\sum_{n=1}^\infty S_b(n)n^{-s}$ has a meromorphic continuation to $\CC$. The poles of $G_b(s)$ consist of a double pole at $s=2$ with Laurent expansion
\begin{equation}
G_b(s) = \frac{b-1}{2\log b}(s-2)^{-2} + \biggl(\frac{b-1}{2\log b}\bigl(\log(2\pi)-1\bigr)-\frac{b+1}{4}\biggr)(s-2)^{-1} + O(1),
\end{equation}
a simple pole at $s=1$ with residue
\begin{equation}
\Res(G_b(s),s=1) = \frac{b+1}{12},
\end{equation}
simple poles at $s=2 + 2\pi i m / \log b$ with $m\in\ZZ$ ($m\neq 0$) with residue
\begin{equation}
\Res\biggl( G_b(s) , s= 2 + \frac{2\pi i m}{\log b} \biggr) = - \frac{b-1}{2\pi i m}\biggl(1+\frac{2\pi i m}{\log b}\biggr)^{-1} \zeta\biggl(\frac{2\pi i m}{\log b}\biggr)
\end{equation}
and simple poles at point $s=2-k + 2\pi i m / \log b$ with $k\geq 2$ an even integer and $m\in\ZZ$ with residue
\begin{equation}
\Res\biggl( G_b(s) , s= 2 - k + \frac{2\pi i m}{\log b} \biggr) = \frac{b-1}{\log b} \zeta\biggl(\frac{2\pi i m}{\log b}\biggr) \biggl(\frac{B_k}{k(k-2)!} \biggr) \prod_{j=1}^{k-2}\biggl(\frac{2\pi i m}{\log b} - j\biggr).
\end{equation}
\end{thm}

An interesting feature of the above theorems is the abundance of poles. Since each function $F_b(s)$ and $G_b(s)$ has $\asymp r^{2}$ poles in the disk $\lvert s \rvert<r$, we have the following corollary, which we discuss further in Section \ref{sec-meromorphic-functions}.

\begin{cor}
The functions $F_b(s)$ and $G_b(s)$ are meromorphic functions of order at least $2$ on $\CC$.
\end{cor}

The Riemann zeta function, the Dirichlet L-functions, and the Dirichlet series generating functions of many important arithmetic functions (such as the M\"obius function $\mu(n)$, the von Mangoldt function $\Lambda(n)$, the Euler totient function $\phi(n)$, and the sum-of-diviors functions $\sigma_\alpha(n)$) analytically continue as meromorphic functions of order $1$ on the complex plane. The Dirichlet series $F_b(s)$ and $G_b(s)$ thus have a different analytic character than many other Dirichlet series considered in number theory.

In Section \ref{sec-non-int}, we use a formula of Delange \cite{delange-75} for $S_b(n)$ to define continuous real-valued interpolations of the functions $d_b(n)$ and $S_b(n)$ from integer bases $b\geq 2$ to a real parameter $\beta>1$. As before, we associate to these interpolated sum-of-digits functions the Dirichlet series
\begin{equation}
F_\beta(s) = \sum_{n=1}^\infty \frac{d_\beta(n)}{n^s}
\end{equation}
and
\begin{equation}
G_\beta(s) = \sum_{n=1}^\infty \frac{S_\beta(n)}{n^s}.
\end{equation}

We prove that these Dirichlet series each have a meromorphic continuation one unit to the left of their halfplane of absolute convergence. For the function $F_\beta(s)$ we have the following theorem.

\begin{thm}
For each real $\beta>1$, the function $F_\beta(s)$ has a meromorphic continuation to the halfplane $\Re(s)>0$, with a double pole at $s=1$ with Laurent expansion
\begin{equation}
F_\beta(s) = \frac{\beta - 1}{2\log \beta}(s-1)^{-2} + \biggl(\frac{\beta-1}{2\log \beta} \bigl(\log(2\pi)\bigr) - \frac{\beta+1}{4}\biggr) (s-1)^{-1} +O(1)
\end{equation}
and simple poles at $s=1+2\pi i m / \log \beta$ for $m\in\ZZ$ with $m\neq 0$ with residue
\begin{equation}
\Res\biggl(F_\beta(s),s=1+\frac{2\pi i m}{\log \beta}\biggr) =  -\frac{\beta-1}{2\pi i m}\zeta\biggl(\frac{2\pi i m}{\log \beta} \biggr).
\end{equation}
\end{thm}

For the function $G_\beta(s)$ we have the following theorem.

\begin{thm}
For each real $\beta>1$, the function $G_\beta(s)$ is meromorphic in the region $\Re(s)>1$ with a double pole at $s=2$ with Laurent expansion
\begin{equation}
G_\beta(s) = \frac{\beta-1}{2\log \beta}(s-2)^{-2} + \biggl(\frac{\beta-1}{2\log \beta} \bigl(\log(2\pi) - 1\bigr) - \frac{\beta+1}{4}\biggr)(s-2)^{-1} +O(1)
\end{equation}
and simple poles at $s=2+2\pi i m / \log \beta$ for $m\in\ZZ$ with $m\neq 0$ with residue
\begin{equation}
\Res\biggl(G_b(s),s=2+\frac{2\pi i m}{\log \beta}\biggr) = -\frac{\beta-1}{2\pi i m}\biggl(1+\frac{2\pi i m}{\log \beta}\biggr)^{-1}\zeta\biggl(\frac{2\pi i m}{\log \beta} \biggr).
\end{equation}
\end{thm}

To prove these theorems, we start by obtaining the continuation of the series $G_\beta(s)$ by working directly with its Dirichlet series and then obtain the continuation of $F_\beta(s)$ by studying the relation between these two Dirichlet series.

\subsection{Previous work}\label{sec-previous-work}

There has been much previous work studying the functions $d_b(n)$ and $S_b(n)$. The  function $d_b(n)$ exhibits significant fluctuations as $n$ changes to $n+1$. It can only increase slowly, having  $d_b(n+1) \leq d_b(n)+1$ but it can decrease by an arbitrarily large amount. The sequence $d_b(n)$ is a $b$-regular sequence in the sense of Allouche and Shallit \cite[Ex.~2, Sec.~7]{allouche-shallit-92} and is a member of the $b$-th arithmetic fractal group
$\Gamma_b(\ZZ)$ of Morton and Mourant \cite[p.~256]{morton-mourant-89}.  Chen et al.\ \cite{chen-hwang-zacharovas-14}  survey results  on the sum-of-digits function of random integers, and give many references.

Concerning the  cumulative sum-of-digits  function, 
Mirsky \cite{mirsky-49} proved in 1949 that for any integer base $b\geq 2$, the function $S_b(n)$ has the asymptotic
\begin{equation}
S_b(n) = \frac{b-1}{2\log b} n \log n + O(n).
\end{equation}
In 1968 Trollope \cite{trollope-68} expressed the error term for the base-$2$ cumulative digit sum $S_2(n)$ in terms of a continuous everywhere nondifferentiable function, the Takagi function---see \cite{lagarias-12} for a survey of the properties of this function. 
In 1975 Delange \cite{delange-75} proved the following formula for $S_b(n)$, expressing the error term as a Fourier series with coefficients involving values of the Riemann zeta function on the imaginary axis.

\begin{thm}[Delange \cite{delange-75}]\label{thm-delange}
The cumulative sum-of-digits function $S_b(n)$ satisfies
\begin{equation}\label{eq-delange-formula}
S_b(n) = \frac{b-1}{2\log b} n \log n + h_b\biggl(\frac{\log n}{\log b} \biggr) n
\end{equation}
where $h_b$ is a nowhere-differentiable function of period $1$. The function $h_b$ has a Fourier series
\begin{equation}
h_b(x) = \sum_{k=-\infty}^\infty c_b(k)e^{2\pi i k x}
\end{equation}
with coefficients
\begin{equation}
c_b(k) = -\frac{b-1}{2\pi i k}\biggl(1+\frac{2\pi i k}{\log b}\biggr)^{-1}\zeta\biggl(\frac{2\pi i k}{\log b} \biggr)
\end{equation}
for $k\neq 0$ and
\begin{equation}
c_b(0) = \frac{b-1}{2\log b}\bigl(\log(2\pi) - 1\bigr) - \frac{b+1}{4}.
\end{equation}
\end{thm}

A complex-analytic proof of a summation formula for general $q$-additive functions, of which the base-$q$ sum-of-digits function is an example, was given by Mauclaire and Murata in 1983 \cite{mauclaire-murata-83a,mauclaire-murata-83b,murata-mauclaire-88}. A shorter complex-analytic proof of \eqref{eq-delange-formula} in the specific case of $S_2(n)$ was given by Flajolet, Grabner, Kirschenhofer, Prodinger, and Tichy in 1994 \cite{flajolet-94}. The method of Flajolet et al.\ is based on applying a variant of Perron's formula to the Dirichlet series
\begin{equation}
\sum_{n=1}^\infty \bigl(d_2(n)-d_2(n-1)\bigr) n^{-s}.
\end{equation}
Grabner and Hwang \cite{grabner-hwang-05} study higher moments of the sum-of-digits function by similar complex-analytic methods.

Our formulas for the residues of $F_b(s)$ and $G_b(s)$ involve the Bernoulli numbers. Kellner \cite{kellner-17} and Kellner and Sondow \cite{kellner-sondow-17} investigate another relation between sums of digits and Bernoulli numbers, proving that the least common multiple of the denominators of the coefficients of the polynomial $\sum_{i=0}^n n^k$, which can be written in terms of a Bernoulli polynomial, can be expressed as a certain product of primes satisfying $d_p(n+1)\geq p$.

\section{Sum-of-digits Dirichlet series}

First we consider basic properties of the Dirichlet series
\begin{equation}
F_b(s) = \sum_{n=1}^\infty \frac{d_b(n)}{n^s}
\end{equation}
attached to the base-$b$ digit sum of $n$ and the Dirichlet series
\begin{equation}
G_b(s) = \sum_{n=1}^\infty \frac{S_b(n)}{n^s}
\end{equation}
attached to the cumulative base-$b$ digit sum. For standard references on the basic analytic properties of Dirichlet series, see the books of Hardy and Riesz \cite{hardy-riesz} or Titchmarsh \cite[Ch.~IX]{titchmarsh-tof}.


Recall that each ordinary Dirichlet series $\sum a_n n^{-s}$ has an abscissa of conditional convergence $\sigma_c$ such that the Dirichlet series converges and defines a holomorphic function if $\Re(s)>\sigma_c$ and diverges if $\Re(s)<\sigma_c$. Each Dirichlet series also has an abscissa of absolute convergence $\sigma_a$ such that the Dirichlet series converges absolutely if $\Re(s)>\sigma_a$ and does not converge absolutely if $\Re(s)<\sigma_a$. For ordinary Dirichlet series, one always has $\sigma_a-1\leq \sigma_c \leq \sigma_a$, and $\sigma_c=\sigma_a$ if the coefficients $a_n$ are nonnegative reals.

\begin{prop}
For each integer $b\geq 2$, the Dirichlet series
\begin{equation}\label{eq-fb-def-2}
F_b(s) = \sum_{n=1}^\infty \frac{d_b(n)}{n^s}
\end{equation}
converges and defines a holomorphic function for $\Re(s)>1$. 
\end{prop}

\begin{proof}
A positive integer $n$ has $[\log n / \log b]+1$ digits when written in base $b$, each of which is at most $b-1$, so
\begin{equation}\label{eq-db-estimate}
d_b(n) \leq (b-1) \biggl(\biggl\lfloor\frac{\log n}{\log b} \biggr\rfloor+1\biggr).
\end{equation}
We then obtain the estimate
\begin{equation}\label{eq-sb-estimate}
 S_b(n) \ll n \log n
\end{equation}
with an implied constant depending on $b$. This implies that the Dirichlet series \eqref{eq-fb-def-2} has abscissa of absolute convergence at most 1 and therefore defines a holomorphic function for $\Re(s)>1$.
\end{proof}

\begin{prop}
For each integer $b\geq 2$, the Dirichlet series
\begin{equation}\label{eq-gb-def-2}
G_b(s) = \sum_{n=1}^\infty \frac{S_b(n)}{n^s}
\end{equation}
converges and defines a holomorphic function for $\Re(s)>2$. 
\end{prop}

\begin{proof}
The estimate \eqref{eq-sb-estimate} gives
\begin{equation}
\sum_{m=1}^n S_b(m) \ll n^2 \log n,
\end{equation}
which shows that the Dirichlet series \eqref{eq-gb-def-2} converges for $\Re(s)>2$.
\end{proof}

It follows from Delange's formula \eqref{eq-delange-formula} that $F_b(s)$ and $G_b(s)$ have abscissa of absolute convergence $\Re(s)=1$ and $\Re(s)=2$, respectively, and this can be proven directly using a more careful estimate of the function $d_b(n)$. We can also obtain the exact values of the abscissas of convergence as a corollary of our theorems on the meromorphic continuation of $F_b(s)$ and $G_b(s)$, since $F_b(s)$ has a pole at $s=1$ and $G_b(s)$ has a pole at $s=2$.

As in previous work on Dirichlet series associated to $q$-additive sequences, it is advantageous to consider the Dirichlet series
\begin{equation}
Z_b(s) =\sum_{n=1}^\infty \bigl(d_b(n)-d_b(n-1)\bigr) n^{-s}
\end{equation}
obtained by differencing the coefficients of the series $F_b(s)$, setting $d_b(0)=0$.
Identity \eqref{eq-zb-formula} in the following proposition appears in a more general form (for $q$-additive functions) in the work of Mauclaire and Murata \cite{mauclaire-murata-83a,mauclaire-murata-83b,murata-mauclaire-88} and is stated and proved explicitly for the sum-of-digits series by Allouche and Shallit \cite{allouche-shallit-90}. We give a more direct proof of this result.

\begin{prop}\label{prop-differenced-zeta}
For each integer $b\geq 2$, the Dirichlet series $Z_b(s)$ has abscissa of absolute convergence $\sigma_a=1$, abscissa of conditional convergence $\sigma_c=0$, and has a meromorphic continuation to $\CC$, satisfying
\begin{equation}\label{eq-zb-formula}
Z_b(s) = \frac{b^s-b}{b^s-1}\zeta(s).
\end{equation}
\end{prop}
\begin{proof}
For bases $b\geq 3$, we have $\lvert d_b(n)-d_b(n-1) \rvert \geq 1$ for all $n$; if $b=2$, we have $\lvert d_b(n)-d_b(n-1)\rvert \geq 1$ for at least all odd $n$. Hence $\sigma_a\geq 1$. We also have $d_b(n)-d_b(n-1)\ll \log n$, so $\sigma_a\leq 1$. The abscissa of conditional convergence $\sigma_c=0$ follows from the bound
\begin{equation}
\sum_{m\leq n} \bigl(d_b(m)-d_b(m-1)\bigr) = d_b(n) \ll \log n.
\end{equation}

The effect of adding 1 on the digit sum in base-$b$ arithmetic depends on the divisibility of $n$ by $b$; in particular, we have
\begin{equation}
d_b(n) - d_b(n-1) = 1 - k(b-1)
\end{equation}
where $k$ is the largest integer such that $b^k \mid n$. We may also express this as
\begin{equation}
d_b(n) - d_b(n-1) = \sum_{m\mid n} \alpha(m) \beta(m/n)
\end{equation}
where 
\begin{equation}
\alpha(n) = \begin{cases} 1 &\text{if $n=b^k$ for some $k$}\\ 0 &\text{otherwise} \end{cases} \qquad \beta(n)= \begin{cases} 1-b &\text{if $b\mid n$} \\ 1 &\text{if $b \nmid n$} \end{cases}.
\end{equation}
Then we have, for $\Re(s)>1$,
\begin{equation}
Z_b(s) = \sum_{n=1}^\infty \bigl( d_b(n) - d_b(n-1) \bigr) n^{-s} =  \sum_{n=1}^\infty \Bigl( \sum_{m\mid n} \alpha(m) \beta(m/n) \Bigr) n^{-s}
\end{equation}
Writing the right side as a product of two Dirichlet series, we have
\begin{equation}
Z_b(s) = \sum_{n=1}^\infty \alpha(n) n^{-s} \sum_{n=1}^\infty \beta(n) n^{-s} = \sum_{n=1}^\infty b^{-ns} \sum_{n=1}^\infty (1-b) (bn)^{-s}.
\end{equation}
Summing the geometric series, we obtain
\begin{equation}
Z_b(s) = \frac{1}{1-b^{-s}} \bigl( \zeta(s) - bb^{-s} \zeta(s)\bigr) =  \frac{b^s-b}{b^s-1}\zeta(s)
\end{equation}
as claimed. Equation \eqref{eq-zb-formula} then provides a meromorphic continuation of $Z_b(s)$ since the right side is meromorphic on $\CC$.
\end{proof}

We will obtain information about the mermorphic continuation of $F_b(s)$ and $G_b(s)$ by considering the relation between these series and the series $Z_b(s)$. For future use, we list the poles of the function $Z_b(s)$.

\begin{lem}\label{lem-poles-of-ndb-ds}
The function $Z_b(s)$ is meromorphic on $\CC$, with simple poles at $s=2\pi i m / \log b$ for $m\in\ZZ$. The residue at each pole is
\begin{equation}
\Res\biggl(Z_b(s),s=\frac{2\pi i m}{\log b}\biggr) = - \frac{b-1}{\log b} \zeta\biggl(\frac{2\pi i m}{\log b}\biggr).
\end{equation}
In particular, at $s=0$, the function $Z_b(s)$ has a Laurent expansion beginning
\begin{equation}
Z_b(s) = \biggl(\frac{b-1}{2\log b}\biggr)s^{-1} + \biggl(- \frac{b+1}{4} + \frac{b-1}{2\log b} \log(2\pi) \biggr) + O(s).
\end{equation}
\end{lem}
\begin{proof}
The function $(b^s-b)/(b^s-1)$ has simple poles at $s=2\pi i m /\log b$ for each $m\in\ZZ$, with residue
\begin{equation}
\Res\biggl(\frac{b^s-b}{b^s-1},s=\frac{2\pi i m}{\log b}\biggr) = -\frac{b-1}{\log b}.
\end{equation}
The Laurent expansion at $s=0$ follows from multiplying the expansions
\begin{equation}
\frac{b^s-b}{b^s-1} = -\frac{b-1}{\log b}s^{-1} + \frac{b+1}{2} + O(s)
\end{equation}
and
\begin{equation}
\zeta(s) = -\frac{1}{2} -\frac{1}{2}\log(2\pi)s + O(s^2).
\end{equation}
The function $\zeta(s)$ has only a simple pole at $s=1$, cancelled by a zero of $(b^s-b)$.
\end{proof}

\section{Meromorphic continuation of $F_b(s)$ and $G_b(s)$}\label{sec-mero-cont}

In this section, we show that for integers $b\geq 2$, the Dirichlet series $F_b(s)$ and $G_b(s)$ have a meromorphic continuation to $\CC$ and determine the structure of the poles, proving Theorems \ref{thm-db} and \ref{thm-sb}.

\subsection{Bernoulli numbers}
Our formulas for the meromorphic continuation of $F_b(s)$ and $G_b(s)$ involve Bernoulli numbers. For standard facts about the Bernoulli numbers and their basic properties, see \cite[Ch.~ 23]{abramowitz-stegun}. For a thorough reference on Bernoulli numbers, their history, and their relation to zeta functions, see \cite{arakawa-et-al}.

The Bernoulli numbers $B_k$ are the sequence of rational numbers defined by the generating function
\begin{equation}\label{eq-bernoulli-gf}
\frac{x}{e^x-1} = \sum_{k=0}^\infty \frac{B_k}{k!}x^k.
\end{equation}
If $x$ is a complex variable, this series converges for $\lvert x \rvert < 2\pi$.

Note that there are several competing notations for the Bernoulli numbers; with our definition, we have $B_0=1$, $B_1=-1/2$, and $B_2=1/6$. Because the function
\begin{equation}
\frac{x}{e^x-1} + \frac{1}{2}x
\end{equation}
is an even function, we find that $B_{2k+1}=0$ for all $k\geq 1$.

\subsection{Power series and Dirichlet series} To prove the meromorphic continuation of $F_b(s)$ and $G_b(s)$, we make use of the following classical relation between Dirichlet series and power series.

\begin{prop}\label{prop-ps-ds-relation}
Let $\sigma_c$ be the abscissa of conditional convergence of the Dirichlet series $\sum_{n=1}^\infty a_n n^{-s}$. Then
\begin{equation}
\Gamma(s) \sum_{n=1}^\infty a_n n^{-s} = \int_0^\infty \Bigl( \sum_{n=1}^\infty a_n e^{-nx} \Bigr) x^{s-1} \, dx
\end{equation}
for $\Re(s) >\max(\sigma_c,0)$.
\end{prop}

\begin{proof}
See \cite[eq.~5.23]{montgomery-vaughan}.
\end{proof}

Proposition \ref{prop-ps-ds-relation} allows us to translate the additive relations between the arithmetic functions $d_b(n)-d_b(n-1)$, $d_b(n)$, and $S_b(n)$, which are easily expressed in terms of power series generating functions, into relations between their associated Dirichlet series.

\subsection{Meromorphic continuation of $F_b(s)$} We now prove the meromorphic continuation of the Dirichlet series $F_b(s)$ by combining the relation between the Dirichlet series and power series generating functions of $d_b(n)$ with the relation between the power series generating functions of $d_b(n)$ and $d_b(n)-d_b(n-1)$.

\begin{proof}[Proof of Theorem \ref{thm-db}]
Let
\begin{equation}\label{eq-px-def}
p(x) = \sum_{n=1}^\infty \bigl(d_b(n)-d_b(n-1)\bigr) x^n.
\end{equation}
We note that
\begin{equation}
\sum_{n=1}^\infty d_b(n)x^n = \frac{p(x)}{1-x}.
\end{equation}
Then by Proposition \ref{prop-ps-ds-relation}, we have
\begin{equation}
\Gamma(s) F_b(s) = \int_0^\infty \frac{1}{1-e^{-x}} p(e^{-x})x^{s-1}\, dx
\end{equation}
for $\Re(s)>1$. The series expansion
\begin{equation}
\frac{x}{1-e^{-x}} = \sum_{k=0}^\infty \frac{(-1)^kB_k}{k!} x^k,
\end{equation}
which follows from \eqref{eq-bernoulli-gf}, holds for $\lvert x \rvert<2\pi$. Since
\begin{equation}
\Gamma(s)Z_b(s) = \int_0^\infty p(e^{-x})x^{s-1}\, dx,
\end{equation}
for $\Re(s)>0$, we can write
\begin{equation}\label{eq-first-fb-expansion}
F_b(s) = \sum_{k=0}^K \frac{(-1)^kB_k}{k!} \frac{\Gamma(s-1+k)}{\Gamma(s)} Z_b(s-1+k) + R_K(s)
\end{equation}
with
\begin{equation}\label{eq-RK-def}
R_K(s) = \frac{1}{\Gamma(s)}\int_0^\infty \Bigl(\frac{x}{1-e^{-x}} - \sum_{k=0}^K \frac{(-1)^k B_k}{k!} x^k \Bigr) p(e^{-x})x^{s-2} \, dx.
\end{equation}
Note that
\begin{equation}
\frac{\Gamma(s-1)}{\Gamma(s)} = \frac{1}{s-1}
\end{equation}
and
\begin{equation}
\frac{\Gamma(s-1+k)}{\Gamma(s)} =  (s)(s+1)\dotsm (s+k-2).
\end{equation}
Since
\begin{equation}
\frac{x}{1-e^{-x}} - \sum_{k=0}^K \frac{(-1)^k B_k}{k!} x^k \ll x^{K+1}
\end{equation}
as $x\rightarrow 0^+$, the integral in \eqref{eq-RK-def} converges and defines a holomorphic function in the region $\Re(s)>1-K$. From Lemma \ref{lem-poles-of-ndb-ds} we know that $Z_b(s)$ has simple poles at $s=2\pi ik/\log b$ for $k\in\ZZ$. The $k=0$ term
\begin{equation}
\frac{1}{s-1}Z_b(s-1)
\end{equation}
has a double pole at $s=1$ with Laurent expansion beginning
\begin{equation}
\frac{1}{s-1}Z_b(s-1) = \frac{b-1}{2\log b}(s-1)^{-2} + \biggl( \frac{b-1}{2\log b} \log(2\pi) - \frac{b+1}{4}\biggr)(s-1)^{-1} + O(1),
\end{equation}
and simple poles at each other point $s=1+2\pi i m /\log b$. Each term
\begin{equation}
(-1)^k\frac{B_k}{k!}\prod_{j=0}^{k-2}(s+j) \cdot Z_b(s-1+k)
\end{equation}
with $k=1$ or with $k$ an even integer with $k\geq 2$ has a simple pole at $s=1-k+2\pi i m/\log b$ for $m\in\ZZ$ with residue
\begin{equation}
(-1)^{k}\frac{B_k}{k!} \prod_{j=1}^{k-1} \biggl(\frac{2\pi i m}{\log b} - j\biggr) \cdot \biggl(-\frac{b-1}{\log b}\biggr)\zeta\biggl(\frac{2\pi i m}{\log b}\biggr).
\end{equation}
Since $K$ can be taken arbitrarily large, this proves the theorem.
\end{proof}

\subsection{Meromorphic continuation of $G_b(s)$} We continue the function $G_b(s)$ to the plane in a similar fashion, using the fact that $S_b(n)$ is a double sum of the differences $d_b(n)-d_b(n-1)$ appearing in the series $Z_b(n)$.

\begin{proof}[Proof of Theorem \ref{thm-sb}]
Define $p(x)$ by \eqref{eq-px-def}. We make use of the identity of power series
\begin{equation}
\frac{x}{(1-x)^2} p(x) = \sum_{n=1}^\infty S_b(n)x^n.
\end{equation}
Then by Proposition \ref{prop-ps-ds-relation}, we have
\begin{equation}
\Gamma(s)G_b(s) = \int_0^\infty \frac{e^{-x}}{(1-e^{-x})^2} p(e^{-x}) x^{s-1} \, dx
\end{equation}
for $\Re(s)>2$. From \eqref{eq-bernoulli-gf} and noting that
\begin{equation}
\frac{e^{-x}}{(1-e^{-x})^2}=-\frac{d}{dx}\biggl(\frac{1}{e^x-1}\biggr),
\end{equation}
we have the power series expansion
\begin{equation}
\frac{x^2e^x}{(e^x-1)^2} = 1-\sum_{k=2}^\infty \frac{B_k}{k(k-2)!}x^k.
\end{equation}
Then for a fixed integer $K\geq 2$ can write $G_b(s)$ as
\begin{equation}\label{eq-gbs-expansion}
G_b(s) = \frac{\Gamma(s-2)}{\Gamma(s)}Z_b(s-2) - \sum_{k=2}^K \frac{B_k}{k(k-2)!} \frac{\Gamma(s-2+k)}{\Gamma(s)} Z_b(s-2+k) + R_K(s)
\end{equation}
with remainder $R_K$ given by
\begin{equation}
R_K(s) = \frac{1}{\Gamma(s)}\int_0^\infty \biggl( \frac{x^2e^{-x}}{(1-e^{-x})^2} - 1 + \sum_{k=2}^K \frac{B_k}{k(k-2)!} x^k \biggr) p(e^{-x}) x^{s-3} \, dx.
\end{equation}
The function $R_K(s)$ is holomorphic for $\Re(s)>2-K$.

As before, we consider the poles of each term of \eqref{eq-gbs-expansion}. The first term
\begin{equation}
\frac{1}{(s-1)(s-2)} Z_b(s-2)
\end{equation}
has a double pole at $s=2$ with Laurent expansion
\begin{equation}
\frac{b-1}{2\log b}(s-2)^{-2} + \biggl(\frac{b-1}{2\log b}\bigl(\log(2\pi)-1\bigr)-\frac{b+1}{4}\biggr)(s-2)^{-1}+ \dotsb,
\end{equation}
a simple pole at each point $s=2+2\pi i m / \log b$ with $m\neq 0$, and a simple pole at $s=1$ with residue $(b+1)/12$. Each other term
\begin{equation}
\frac{B_k}{k(k-2)!} \prod_{j=0}^{k-3} (s+j) \cdot Z_b(s-2+k).
\end{equation}
has simple poles at $s=2-k+2\pi i m / \log b$ for all $m$.
\end{proof}

\begin{rem}
Instead of relating $G_b(s)$ to the series $Z_b(s)$ as we did in this proof, we could have also used the relation
\[ \Gamma(s) G_b(s) = \int_0^\infty \frac{e^{-x}}{1-e^{-x}} \biggl(\sum_{n=1}^\infty d_b(n) e^{-xn}\biggr) x^{s-1} \, dx, \]
following the proof of Theorem \ref{thm-db} to write $G_b(s)$ in terms of $F_b(s)$.
\end{rem}

\subsection{Order of $F_b(s)$ and $G_b(s)$ as meromorphic functions}\label{sec-meromorphic-functions}

The functions $F_b(s)$ and $G_b(s)$ are meromorphic functions on the complex plane with infinitely many poles on a left half-lattice. We now raise a further question about the analytic properties of these functions. Recall that the \emph{order} of an entire function $f(z)$ on $\CC$ is
\[ \rho = \inf \bigl\{ \rho\geq 0 \, \bigm\vert \, f(z) = O\bigl(\exp(\lvert z \rvert^{\rho+\varepsilon})\bigr)\text{ as $\lvert z \rvert\rightarrow\infty$} \bigr\}. \]
An entire function is of \emph{finite order} if $\rho<\infty$.

The order of a meromorphic function is defined as the order of growth of its associated Nevanlinna characteristic function. This definition is equivalent (by \cite{rubel} Lemma 15.6 and Theorem on p.\ 91) to the following.

\begin{defn}
The \emph{order} of a meromorphic function $f(z)$ is the infimum of $\rho\geq 0$ such that $f$ can be written as $f(z) = g(z)/h(z)$ for entire functions $g(z)$ and $h(z)$ of order $\rho$.
\end{defn}

Many of the common Dirichlet series of analytic number theory are meromorphic (or entire) functions of order $1$. Examples include the Riemann zeta function $\zeta(s)$, the Dirichlet $L$-functions $L(s,\chi)$, and their relatives; more generally, all Dirichlet series in the Selberg class (as introduced in Selberg \cite{selberg-92}) are meromorphic functions of order 1.

By Proposition \ref{prop-differenced-zeta}, the function $Z_b(s)$ is meromorphic of order $1$. The functions $F_b(s)$ and $G_b(s)$, however, must have greater order by the following fact.

\begin{prop}[{\cite[p.~284]{titchmarsh-tof}}]
Let $f(z)$ be a meromorphic function of order $\rho$ and let $n(r,a)$ be the number of zeros of $f(z)-a$ in the disc $\lvert z \rvert<r$. Then $n(r,a)=O(r^{\rho+\varepsilon})$.
\end{prop}

By Theorems \ref{thm-db} and \ref{thm-sb}, the functions $F_b(s)$ and $G_b(s)$ each have $\gg r^2$ poles in the disc $\lvert z \rvert<r$. Hence we have the following corollary.

\begin{cor}
The functions $F_b(s)$ and $G_b(s)$ are meromorphic functions of order at least $2$.
\end{cor}

A meromorphic function of order greater than 2 could still have only $O(r^2)$ poles in $\lvert z \rvert<r$, so without further information, we cannot deduce that $F_b(s)$ and $G_b(s)$ have order 2.
\begin{quest}
Are the functions $F_b(s)$ and $G_b(s)$ mermorphic functions of order exactly $2$?
\end{quest}

Such a question has been answered positively in the related setting of strongly $q$-multiplicative functions: the Dirichlet series attached to such functions are entire of order exactly $2$ (see Alkauskas \cite{alkauskas-04}).

\section{Meromorphic continuation of Dirichlet series for non-integer bases}\label{sec-non-int}

In this section, we consider the problem of extending the digit sums $d_b(n)$ and $S_b(n)$ from integer bases $b$ to real parameters $\beta>1$. There are a number of possible ways to do this. One natural approach concerns the notion of $\beta$-expansions introduced by Renyi \cite{renyi-57} and studied at length by Parry \cite{parry-60}. However, for non-integer values of $\beta$, the $\beta$-expansion of an integer generally has infinitely many digits, so the sum of the digits will generally be infinite.
Digit sums related to a different digit expansion with respect to an irrational base were considered by Grabner and Tichy \cite{grabner-tichy-91}.

The approach which we consider in this section is to use the formula of Delange to define the cumulative digit sum $S_\beta(n)$ for real parameters $\beta>1$, from which we can define a digit sum $d_\beta(n)$ by differencing. The resulting functions are continuous in the $\beta$-parameter.

\subsection{Extension to non-integer bases by Delange's formula}

We begin by replacing the integer variable $b$ in Theorem \ref{thm-delange}, which gives a formula for $S_b(n)$ for integer bases $b\geq 2$, by a real parameter $\beta>1$.

\begin{defn}\label{def-non-integer}
For $\beta\in\RR$ with $\beta>1$, define a generalized cumulative sum-of-digits function $S_\beta(n)$ by
\begin{equation}\label{eq-sbeta-delange-def}
S_\beta(n) := \frac{\beta-1}{2\log \beta} n \log n + h_\beta\biggl(\frac{\log n}{\log \beta} \biggr) n,
\end{equation}
where the function $h_\beta(x)$ is defined by the Fourier series
\begin{equation}\label{eq-hbeta-fourier-series}
h_\beta(x) = \sum_{k=-\infty}^\infty c_\beta(k)e^{2\pi i k x}
\end{equation}
with coefficients
\begin{equation}
c_\beta(k) = -\frac{\beta-1}{2\pi i k}\biggl(1+\frac{2\pi i k}{\log \beta}\biggr)^{-1}\zeta\biggl(\frac{2\pi i k}{\log \beta} \biggr)
\end{equation}
for $k\neq 0$ and
\begin{equation}
c_\beta(0) = \frac{\beta-1}{2\log \beta}(\log 2\pi - 1) - \frac{\beta+1}{4}.
\end{equation}
\end{defn}
\begin{defn}
Define the generalized sum-of-digits function $d_\beta(n)$ for real $\beta>1$ by
\begin{equation}
d_\beta(n) := S_\beta(n+1) - S_\beta(n).
\end{equation}
\end{defn}

A plot of $S_\beta(10)$ as a function of $\beta$ for $1\leq\beta\leq 15$ is shown in Figure \ref{Sb10-plot}. Note that $S_\beta(n)$ is approximately constant for $\beta\geq 10$.

\begin{figure}\label{Sb10-plot}
\includegraphics[scale=0.5]{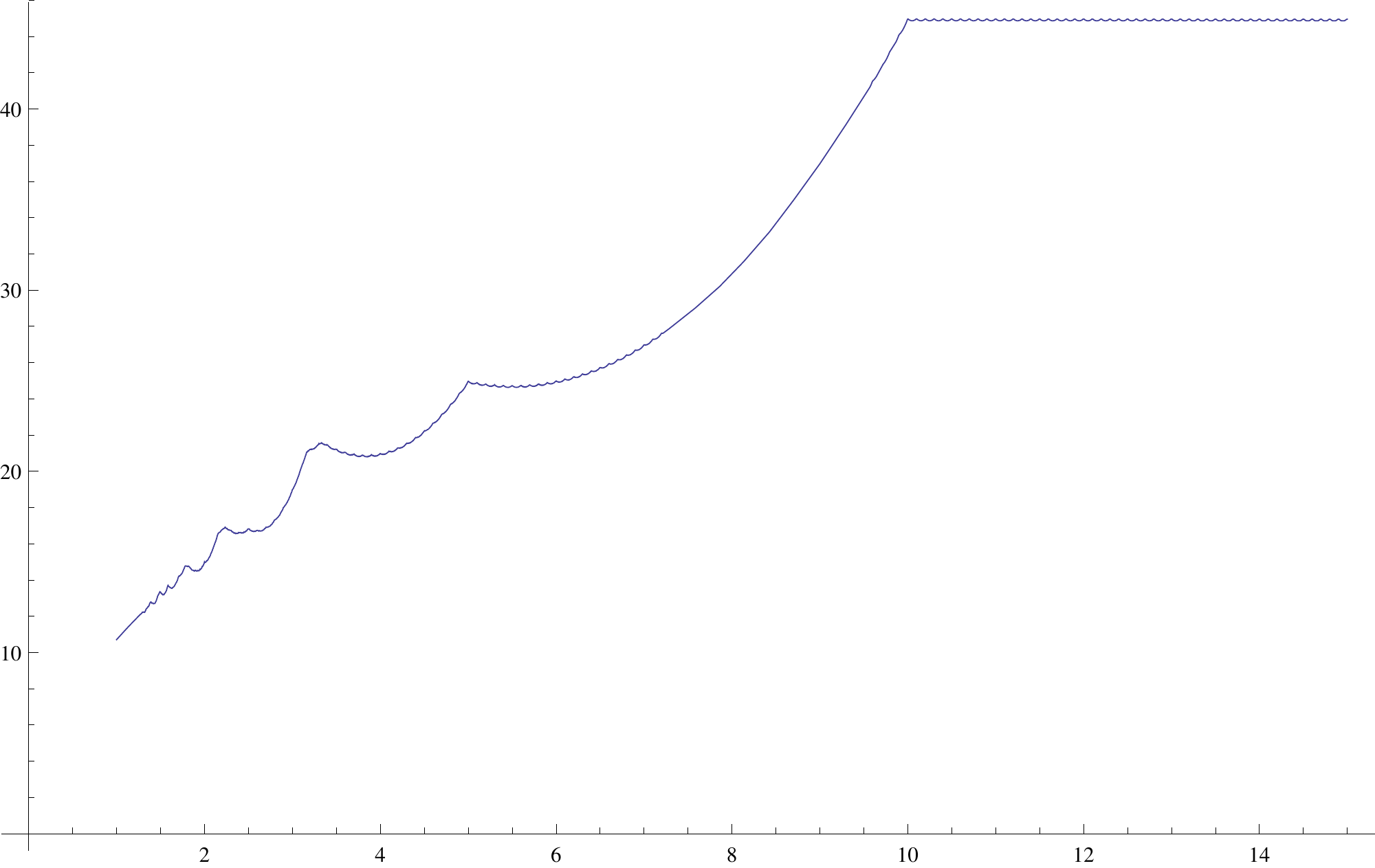}
\caption{A plot of $S_\beta(10)$ for $1\leq\beta\leq 15$, using terms with $\lvert k \rvert\leq 1000$ in the Fourier series for $h_\beta(x)$.}
\end{figure}

\subsection{The function $h_\beta(x)$}

In this section, we study properties of the function $h_\beta(x)$ appearing in Definition \ref{def-non-integer} as a function of the variable $\beta>1$ and as a function of the variable $x$. When $\beta=b\in\NN$, Delange showed that $h_b(x)$ is a continuous but everywhere non-differentiable real-valued function of $x$ with period 1.

\begin{lem}
For each fixed $\beta>1$, the function $h_\beta(x)$ is a real-valued continuous function of $x$ on $\RR$.
\end{lem}
\begin{proof}
The zeta function satisfies the bound
\begin{equation}
\lvert\zeta(it)\rvert\ll t^{1/2+\varepsilon}
\end{equation}
for $t\in\RR$ (see for example \cite[eq.~5.1.3]{titchmarsh-zeta}, so the Fourier coefficients of $h_\beta$ satisfy
\begin{equation}
c_\beta(k)= -\frac{\beta-1}{2\pi i k}\biggl(1+\frac{2\pi i k}{\log \beta}\biggr)^{-1}\zeta\biggl(\frac{2\pi i k}{\log \beta} \biggr) \ll_\beta k^{-3/2+\varepsilon}.
\end{equation}
This estimate shows that the Fourier series \eqref{eq-hbeta-fourier-series} is absolutely and uniformly convergent for $x\in\RR$, so gives a continuous function of $x$.

The function $h_\beta(x)$ is real-valued for $x\in\RR$ since the Fourier coefficients $c_\beta(k)$ satisfy $\overline{c_\beta(k)} = c_\beta(-k)$.
\end{proof}

\begin{figure}\label{fb2-plot}
\includegraphics[scale=0.5]{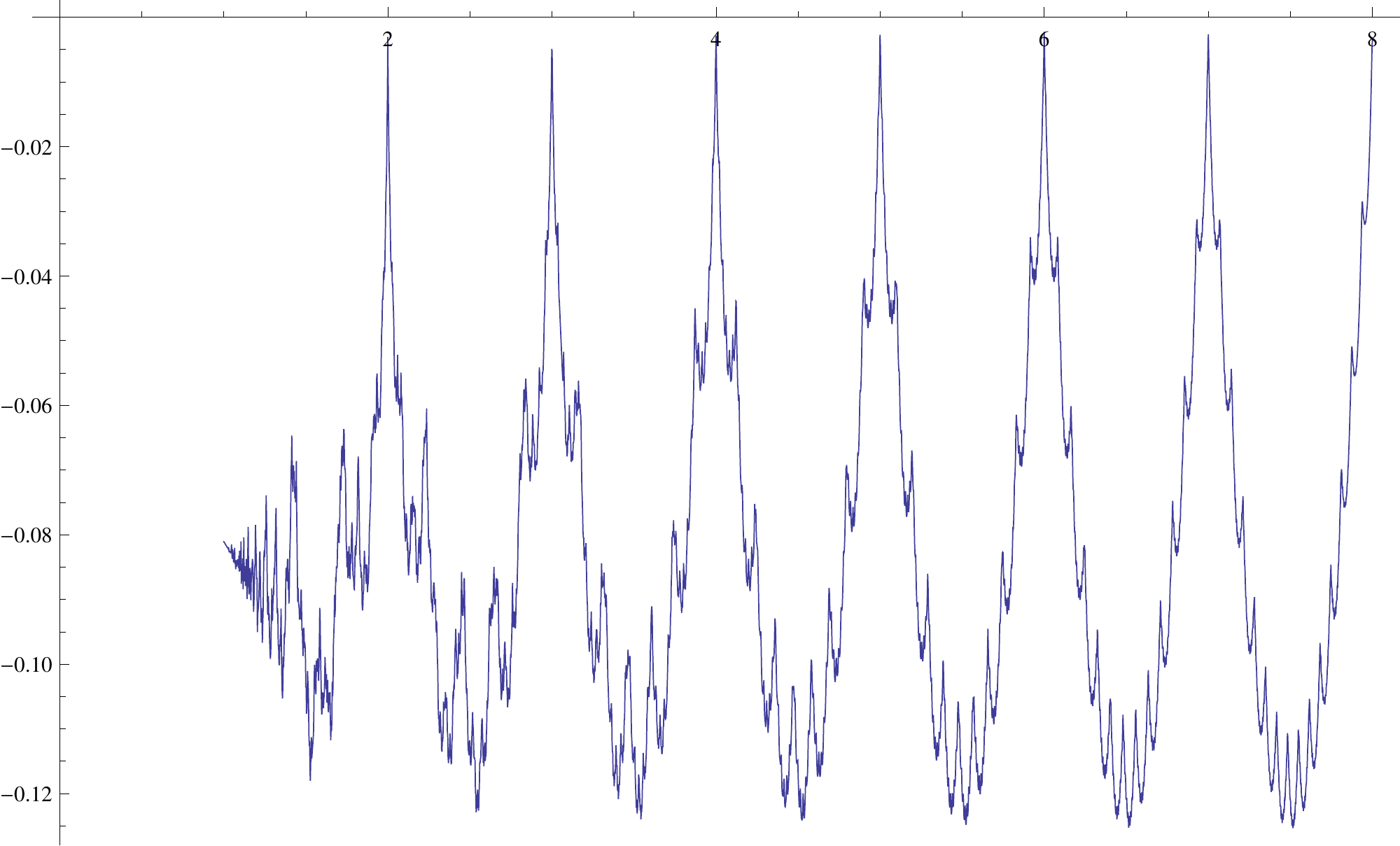}
\caption{A plot of $h_\beta(2)$ for $1\leq\beta\leq 8$, using terms with $\lvert k \rvert\leq 1000$ in the Fourier series for $h_\beta(x)$.}
\end{figure}

\begin{figure}\label{fbl2-plot}
\includegraphics[scale=0.5]{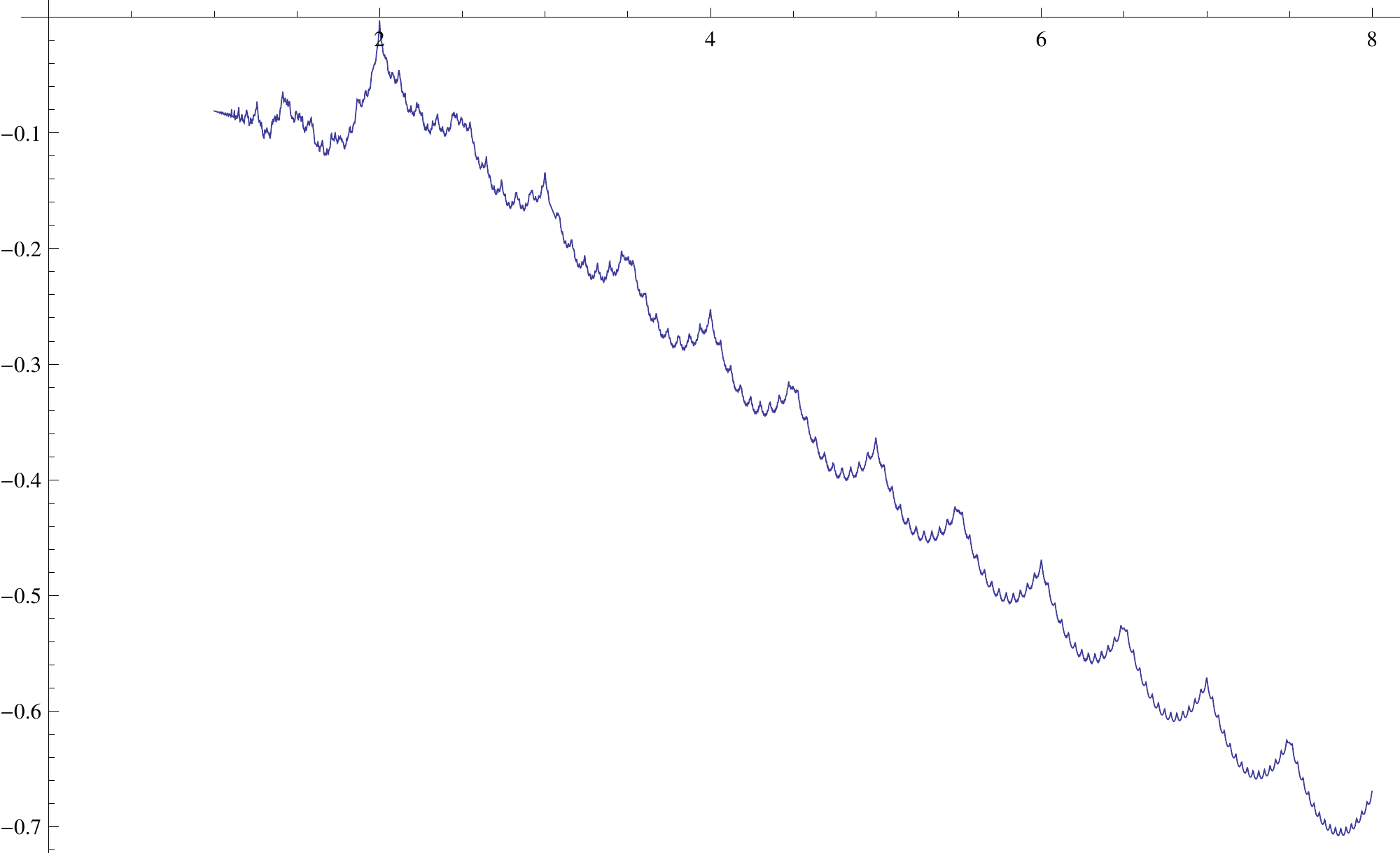}
\caption{A plot of $h_\beta(\log 2 / \log \beta)$ for $1\leq\beta\leq 8$, using terms with $\lvert k \rvert\leq 1000$ in the Fourier series for $h_\beta(x)$.}
\end{figure}

A plot of $h_\beta(2)$ as a function of the real parameter $\beta$ for $1\leq \beta \leq 8$ is shown in Figure \ref{fb2-plot}. From the plot, it also appears that $h_\beta$ might be non-differentiable as a function of the real parameter $\beta$. 

\begin{quest}
For fixed $x\in\RR$, is the function $h_\beta(x)$ everywhere non-differentiable as a function of the real variable $\beta$?
\end{quest}

\subsection{Meromorphic continuation of $G_\beta(s)$}

Our proofs of the meromorphic continuation of $F_b(s)$ and $G_b(s)$ for integer bases relied on the identity
\begin{equation}
Z_b(s) = \sum_{n=1}^\infty \bigl( d_b(n) - d_b(n-1) \bigr) n^{-s} = \frac{b^s-b}{b^s-1}\zeta(s).
\end{equation}
If for non-integer $\beta>1$ we define
\begin{equation}
Z_\beta(s) \coloneq \sum_{n=1}^\infty \bigl( d_\beta(n) - d_\beta(n-1) \bigr) n^{-s},
\end{equation}
then $Z_\beta(s)$ is \emph{not} equal to
\begin{equation}\label{eq-zbeta-wrong}
\frac{\beta^s-\beta}{\beta^s-1}\zeta(s)
\end{equation}
as \eqref{eq-zbeta-wrong} is not an ordinary Dirichlet series. We must therefore take a different approach.

We first consider the Dirichlet series generating function
\begin{equation}
G_\beta(s) := \sum_{n=1}^\infty \frac{S_\beta(n)}{n^s}
\end{equation}
for $\beta\in\RR$ with $\beta>1$. Since the coefficients satisfy
\begin{equation}
S_\beta(n) \asymp n \log n,
\end{equation}
the Dirichlet series $G_\beta(s)$ has abscissa of absolute convergence $\sigma_a=2$. We show that the function $G_\beta(s)$ can be analytically continued to a larger halfplane.

\begin{thm}
For each real $\beta>1$, the function $G_\beta(s)$ is meromorphic in the region $\Re(s)>1$ with a double pole at $s=2$ with Laurent expansion
\begin{equation}
G_\beta(s) = \frac{\beta - 1}{2\log \beta}(s-2)^{-2} + c_\beta(0) (s-2)^{-1} +O(1)
\end{equation}
and simple poles at $s=2+2\pi i k / \log \beta$ for $k\in\ZZ$ with $k\neq 0$ with residue
\begin{equation}
\Res\biggl(G_b(s),s=2+\frac{2\pi i k}{\log \beta}\biggr) = c_\beta(k),
\end{equation}
where the numbers $c_\beta(k)$ are those in Definition \ref{def-non-integer}.
\end{thm}
\begin{proof}
Using the definition \eqref{eq-sbeta-delange-def} of $S_\beta$, we have
\begin{align}
G_\beta(s)&=\sum_{n=1}^\infty \Biggl(\frac{\beta-1}{2\log \beta} n \log n + h_\beta\biggl(\frac{\log n}{\log \beta} \biggr) n\Biggr)n^{-s}\\
&= -\frac{\beta-1}{2\log \beta} \zeta'(s-1) + \sum_{n=1}^\infty h_\beta\biggl(\frac{\log n}{\log \beta} \biggr) n^{-(s-1)}.
\end{align}
The function $\zeta'(s-1)$ is meromorphic on $\CC$ with only singularity a double pole at $s=2$ with Laurent expansion $\zeta'(s-1)=-(s-1)^{-2} + O(1)$. Using the Fourier series \eqref{eq-hbeta-fourier-series} for $h_\beta$, we have
\begin{align}
\sum_{n=1}^\infty h_\beta\biggl(\frac{\log n}{\log \beta} \biggr) n^{-(s-1)} &= \sum_{n=1}^\infty \sum_{k=-\infty}^\infty c_\beta(k) \exp\biggl(2 \pi i k \frac{\log n}{\log \beta}\biggr) n^{-(s-1)}\\ &= \sum_{n=1}^\infty \sum_{k=-\infty}^\infty c_\beta(k) n^{-(s-1 - 2\pi i k / \log \beta)}.
\end{align}
This double sum is absolutely convergent, so we may exchange the sums, giving
\begin{equation}
\sum_{n=1}^\infty h_\beta\biggl(\frac{\log n}{\log \beta} \biggr) n^{-(s-1)} = \sum_{k=-\infty}^\infty c_\beta(k) \zeta\biggl( s- 1 - \frac{2\pi i k}{\log \beta}\biggr).
\end{equation}
If $\Re(s)>1$, then
\begin{equation}
\zeta\biggl( s- 1 - \frac{2\pi i k}{\log \beta}\biggr) \ll k^{1/2+\varepsilon}
\end{equation}
for any $\varepsilon>0$. On any compact set $K$ in the halfplane $\Re(s)>1$ not containing a point $s=2+2\pi i k / \log \beta)$ for any $k\in\ZZ$, the sum
\begin{equation}
\sum_{k=-\infty}^\infty c_\beta(k) \zeta\biggl( s- 1 - \frac{2\pi i k}{\log \beta}\biggr)
\end{equation}
is uniformly convergent on $K$; if the compact set $K$ contains a point of the form $s=2 + 2\pi i k_0 / \log \beta$, then one term of the sum has a simple pole with residue $c_\beta(k_0)$ while the remaining sum is uniformly convergent.
\end{proof}

When $\beta\geq 2$ is an integer, we know that the function $G_\beta(s)$ has a meromorphic continuation to the entire complex plane.

\begin{quest}
For noninteger $\beta>1$, does the Dirichlet series $G_\beta(s)$ have a meromorphic continuation beyond $\Re(s)>1$?
\end{quest}

\subsection{Meromorphic continuation of $F_\beta(s)$}

We now consider the Dirichlet series
\begin{equation}
F_\beta(s) = \sum_{n=1}^\infty \frac{d_\beta(n)}{n^s}
\end{equation}
for real $\beta>1$. We already know that this series has a meromorphic continuation to $\CC$ when $\beta\geq 2$ is an integer. We show that for each real $\beta>1$, the Dirichlet series $F_\beta(s)$ has a meromorphic continuation to the halfplane $\Re(s)>0$.

\begin{thm}
The function $F_\beta(s)$ has a meromorphic continuation to the halfplane $\Re(s)>0$, with a double pole at $s=1$ with Laurent expansion
\begin{equation}
F_\beta(s) = \frac{\beta - 1}{2\log \beta}(s-1)^{-2} + \biggl(c_\beta(0) +\frac{\beta-1}{2\log \beta}\biggr) (s-1)^{-1} +O(1)
\end{equation}
and simple poles at $s=1+2\pi i k / \log \beta$ for $k\in\ZZ$ with $k\neq 0$ with residue
\begin{equation}
\Res\biggl(F_\beta(s),s=1+\frac{2\pi i k}{\log \beta}\biggr) = \biggl(1+\frac{2\pi i k}{\log \beta}\biggr) c_\beta(k).
\end{equation}
\end{thm}

\begin{proof}
Let
\begin{equation}
p(x) = \sum_{n=2}^\infty S_\beta(n)x^n,
\end{equation}
so that
\begin{equation}
\Gamma(s)\bigl(G_\beta(s)-S_\beta(1)\bigr) = \int_0^\infty p(e^{-x}) x^{s-1}\, dx.
\end{equation}
By our definition of $d_\beta(n)$, we have
\begin{equation}
\sum_{n=1}^\infty d_\beta(n) x^n + S_\beta(1)= (x^{-1}-1)p(x).
\end{equation}
Hence by Proposition \ref{prop-ps-ds-relation} we have
\begin{equation}
\Gamma(s)\bigl(F_\beta(s)+S_\beta(1)\bigr) = \int_0^\infty (e^x-1)p(e^{-x})x^{s-1}\, dx
\end{equation}
for $\Re(s)>1$. Using the power series expansion
Then we write
\begin{equation}\label{eq-fbeta-formula-1}
\Gamma(s)\bigl(F_\beta(s)+S_\beta(1)\bigr)  = \Gamma(s+1)\bigl(G_\beta(s+1)-S_\beta(1)\bigr) + \int_0^\infty (e^x-1 - x) p(e^{-x})x^{s-1} \, dx.
\end{equation}
Dividing by $\Gamma(s)$ and rearranging, we obtain
\begin{equation}\label{eq-fbeta-gbeta-relation}
F_\beta(s) = -S_\beta(1)(s+1) + s G_\beta(s+1) + R(s)
\end{equation}
where the remainder term
\begin{equation}
R(s) = \frac{1}{\Gamma(s)} \int_0^\infty (e^x-1 - x) p(e^{-x})x^{s-1} \, dx
\end{equation}
is holomorphic in $\Re(s)>0$ since $e^x-1-x\ll x^2$ as $x\rightarrow 0^+$. Since $G_\beta(s+1)$ is meromorphic in $\Re(s)>0$, we find that $F_\beta(s)$ is meromorphic in $\Re(s)>0$, with poles coming from the poles of $G_\beta(s+1)$. Since
\begin{equation}
sG_\beta(s+1) = (s-1)G_\beta(s+1) + G_\beta(s+1),
\end{equation}
we find that $F_\beta(s)$ has a double pole at $s=1$ with Laurent expansion as given in the theorem. At each other point $s=1+2\pi i m / \log \beta$, $F_\beta(s)$ has a simple pole.
\end{proof}

Meromorphic continuation of $F_\beta(s)$ to a larger halfplane would follow from continuation of $G_\beta(s)$ to a larger halfplane; in particular, by using more terms of the power series for $e^x$ in formula \eqref{eq-fbeta-formula-1}, we find that if $G_\beta(s)$ is meromorphic in $\Re(s)>c$ for some $c$, then $F_\beta(s)$ is meromorphic in $\Re(s)>c-1$.

\section{Acknowledgements}
The author thanks Jeffrey Lagarias for many helpful comments.

\bibliographystyle{amsplain}
\bibliography{sumofdigitsbibliography}

\providecommand{\bysame}{\leavevmode\hbox to3em{\hrulefill}\thinspace}
\providecommand{\MR}{\relax\ifhmode\unskip\space\fi MR }
\providecommand{\MRhref}[2]{%
  \href{http://www.ams.org/mathscinet-getitem?mr=#1}{#2}
}
\providecommand{\href}[2]{#2}
\begin{thebibliography}{10}

\bibitem{abramowitz-stegun}
Milton Abramowitz and Irene~A. Stegun, \emph{Handbook of mathematical functions
  with formulas, graphs, and mathematical tables}, National Bureau of Standards
  Applied Mathematics Series, vol.~55, For sale by the Superintendent of
  Documents, U.S. Government Printing Office, Washington, D.C., 1964.
  \MR{0167642}

\bibitem{alkauskas-04}
Giedrius Alkauskas, \emph{Dirichlet series associated with strongly
  {$q$}-multiplicative functions}, Ramanujan J. \textbf{8} (2004), no.~1,
  13--21. \MR{2068427}

\bibitem{allouche-shallit-90}
Jean-Paul Allouche and Jeffrey Shallit, \emph{Sums of digits and the {H}urwitz
  zeta function}, Analytic number theory ({T}okyo, 1988), Lecture Notes in
  Math., vol. 1434, Springer, Berlin, 1990, pp.~19--30. \MR{1071742}

\bibitem{allouche-shallit-92}
\bysame, \emph{The ring of {$k$}-regular sequences}, Theoret. Comput. Sci.
  \textbf{98} (1992), no.~2, 163--197. \MR{1166363}

\bibitem{arakawa-et-al}
Tsuneo Arakawa, Tomoyoshi Ibukiyama, and Masanobu Kaneko, \emph{Bernoulli
  numbers and zeta functions}, Springer Monographs in Mathematics, Springer,
  Tokyo, 2014, With an appendix by Don Zagier. \MR{3307736}

\bibitem{chen-hwang-zacharovas-14}
Louis H.~Y. Chen, Hsien-Kuei Hwang, and Vytas Zacharovas, \emph{Distribution of
  the sum-of-digits function of random integers: a survey}, Probab. Surv.
  \textbf{11} (2014), 177--236. \MR{3269227}

\bibitem{delange-75}
Hubert Delange, \emph{Sur la fonction sommatoire de la fonction``somme des
  chiffres''}, Enseignement Math. (2) \textbf{21} (1975), no.~1, 31--47.
  \MR{0379414 (52 \#319)}

\bibitem{dumas-thesis}
Philippe Dumas, \emph{R\'ecurrences mahl\'eriennes, suites automatiques,
  \'etudes asymptotiques}, Institut National de Recherche en Informatique et en
  Automatique (INRIA), Rocquencourt, 1993, Th\`ese, Universit\'e de Bordeaux I,
  Talence, 1993. \MR{1346304}

\bibitem{flajolet-94}
Philippe Flajolet, Peter Grabner, Peter Kirschenhofer, Helmut Prodinger, and
  Robert~F. Tichy, \emph{Mellin transforms and asymptotics: digital sums},
  Theoret. Comput. Sci. \textbf{123} (1994), no.~2, 291--314. \MR{1256203}

\bibitem{grabner-hwang-05}
Peter~J. Grabner and Hsien-Kuei Hwang, \emph{Digital sums and
  divide-and-conquer recurrences: {F}ourier expansions and absolute
  convergence}, Constr. Approx. \textbf{21} (2005), no.~2, 149--179.
  \MR{2107936}

\bibitem{grabner-tichy-91}
Peter~J. Grabner and Robert~F. Tichy, \emph{{$\alpha$}-expansions, linear
  recurrences, and the sum-of-digits function}, Manuscripta Math. \textbf{70}
  (1991), no.~3, 311--324. \MR{1089067}

\bibitem{hardy-riesz}
G.~H. Hardy and M.~Riesz, \emph{The general theory of {D}irichlet's series},
  Cambridge Tracts in Mathematics and Mathematical Physics, No. 18,
  Stechert-Hafner, Inc., New York, 1964. \MR{0185094}

\bibitem{kellner-17}
Bernd~C. Kellner, \emph{On a product of certain primes}, J. Number Theory
  \textbf{179} (2017), 126--141. \MR{3657160}

\bibitem{kellner-sondow-17}
Bernd~C. Kellner and Jonathan Sondow, \emph{Power-sum denominators}, Amer.
  Math. Monthly \textbf{124} (2017), no.~8, 695--709. \MR{3706817}

\bibitem{lagarias-12}
Jeffrey~C. Lagarias, \emph{The {T}akagi function and its properties}, Functions
  in number theory and their probabilistic aspects, RIMS K\^oky\^uroku
  Bessatsu, B34, Res. Inst. Math. Sci. (RIMS), Kyoto, 2012, pp.~153--189.
  \MR{3014845}

\bibitem{mauclaire-murata-83a}
J.-L. Mauclaire and Leo Murata, \emph{On {$q$}-additive functions. {I}}, Proc.
  Japan Acad. Ser. A Math. Sci. \textbf{59} (1983), no.~6, 274--276.
  \MR{718620}

\bibitem{mauclaire-murata-83b}
\bysame, \emph{On {$q$}-additive functions. {II}}, Proc. Japan Acad. Ser. A
  Math. Sci. \textbf{59} (1983), no.~9, 441--444. \MR{732606}

\bibitem{mirsky-49}
L.~Mirsky, \emph{A theorem on representations of integers in the scale of
  {$r$}}, Scripta Math. \textbf{15} (1949), 11--12. \MR{0030991}

\bibitem{montgomery-vaughan}
Hugh~L. Montgomery and Robert~C. Vaughan, \emph{Multiplicative number theory.
  {I}. {C}lassical theory}, Cambridge Studies in Advanced Mathematics, vol.~97,
  Cambridge University Press, Cambridge, 2007. \MR{2378655}

\bibitem{morton-mourant-89}
Patrick Morton and W.~J. Mourant, \emph{Paper folding, digit patterns and
  groups of arithmetic fractals}, Proc. London Math. Soc. (3) \textbf{59}
  (1989), no.~2, 253--293. \MR{1004431}

\bibitem{murata-mauclaire-88}
Leo Murata and Jean-Loup Mauclaire, \emph{An explicit formula for the average
  of some {$q$}-additive functions}, Prospects of mathematical science
  ({T}okyo, 1986), World Sci. Publishing, Singapore, 1988, pp.~141--156.
  \MR{948466}

\bibitem{parry-60}
W.~Parry, \emph{On the {$\beta $}-expansions of real numbers}, Acta Math. Acad.
  Sci. Hungar. \textbf{11} (1960), 401--416. \MR{0142719}

\bibitem{renyi-57}
A.~R\'enyi, \emph{Representations for real numbers and their ergodic
  properties}, Acta Math. Acad. Sci. Hungar \textbf{8} (1957), 477--493.
  \MR{0097374}

\bibitem{rubel}
Lee~A. Rubel, \emph{Entire and meromorphic functions}, Universitext,
  Springer-Verlag, New York, 1996, With the assistance of James E. Colliander.
  \MR{1383095}

\bibitem{selberg-92}
Atle Selberg, \emph{Old and new conjectures and results about a class of
  {D}irichlet series}, Proceedings of the {A}malfi {C}onference on {A}nalytic
  {N}umber {T}heory ({M}aiori, 1989), Univ. Salerno, Salerno, 1992,
  pp.~367--385. \MR{1220477}

\bibitem{titchmarsh-tof}
E.~C. Titchmarsh, \emph{The theory of functions}, Oxford University Press,
  Oxford, 1958, Reprint of the second (1939) edition. \MR{3155290}

\bibitem{titchmarsh-zeta}
\bysame, \emph{The theory of the {R}iemann zeta-function}, second ed., The
  Clarendon Press, Oxford University Press, New York, 1986, Edited and with a
  preface by D. R. Heath-Brown. \MR{882550}

\bibitem{trollope-68}
J.~R. Trollope, \emph{An explicit expression for binary digital sums}, Math.
  Mag. \textbf{41} (1968), 21--25. \MR{0233763 (38 \#2084)}

\end{thebibliography}

\end{document}